\documentclass[12pt,a4paper,twoside]{article}
\usepackage{amsmath}
\usepackage{amsthm}
\usepackage{amssymb}
\usepackage{verbatim}
\usepackage{booktabs}
\usepackage{caption}
\usepackage{subfig}
\usepackage{tabulary}

\usepackage[pdftex]{graphicx}

\DeclareMathOperator*{\argmax}{argmax}
\DeclareMathOperator*{\argmin}{argmin}

\newcommand{\R}{{\Bbb R}}

\title{Finding maxmin allocations in cooperative and competitive fair division}
\author{Marco Dall'Aglio \\ LUISS University \\Rome, Italy \\ \texttt{mdallaglio@luiss.it} \and
Camilla Di Luca \\
LUISS University \\Rome, Italy \\
\texttt{cdiluca@luiss.it} }

\date{First draft: September 28,2011 \\
This revision: August 10, 2012}

\theoremstyle{plain}
\newtheorem{teo}{Theorem}[section]
\newtheorem{defin}[teo]{Definition}
\newtheorem{prop}[teo]{Proposition}

\theoremstyle{remark}

\begin{document}

\maketitle

\begin{abstract}
We consider upper and lower bounds for maxmin allocations of  a completely divisible good in both competitive and cooperative strategic contexts. These bounds are based on the convexity properties of the range of utility vectors associated to all possible divisions of the good. We then derive a subgradient algorithm to compute the exact value up to any fixed degree of precision.
\end{abstract}

\section{Introduction}
The notion of what is fair in the allocation of one or more infinitely divisible goods to a finite number of agents with their own preferences has long been debated. Predictably, no agreement has been reached on the subject. The situation is often exemplified with children (players) at a birthday party who are around a table waiting for their slice of the cake to be served, with the help of some parent (an impartial referee). If we think about a special class of resolute children who are able to specify their preferences in terms of utility set functions, the parent in charge of the division could ease his task by using a social welfare function to summarize the children's utility values. Among the many proposals, the maxmin -- or Rawlsian -- division was extensively studied in the seminal work of Dubins and Spanier \cite{ds61}, who showed the existence of maxmin optimal partitions of the cake for any completely divisible cake and their main properties. They also showed that when a condition of mutual appreciation holds (assumption {\em (MAC)} below)
 any optimal partition is also {\em equitable}, i.e., it assigns the same level of utility for each child.

The study of the maxmin optimal partitions and their properties has continued in more recent years. In particular, its relationship with other important notions such as efficiency (or Pareto optimality) and, above all, envy-freeness has been investigated with alternating success: each maxmin partition is efficient, but while for the two children case Brams and Taylor \cite{bt96} showed that it is also envy-free, the same may not hold when three or more children are to be served, as shown in Dall'Aglio and Hill \cite{dh03}.

It is worth pointing out the relationship with $n$ player bargaining solutions. If we think about the division as deriving from a bargaining procedure among children, it is straightforward to show that the bargaining solution proposed by Kalai \cite{k77}, in the case where all the players' utilities are normalized to 1, coincides with the equitable maxmin division. Therefore, if the conditions proposed by Dubins and Spanier hold, the two solutions actually coincide.

Little attention has been devoted, however, to finding optimal maxmin partitions with one notable exceptions: the case of two players with additive and linear utility over several goods has been considered by Brams and Taylor \cite{bt96}, with the Adjusted Winner procedure.

For the case of general preferences (expressed as probability measures, i.e. nonnegative and countably additive set functions normalized to 1) and arbitrary number of players, Legut and Wilczinski \cite{lw88} gave a characterization of the optimal maxmin allocation in terms of weighted density functions. Moreover, Elton et al. \cite{ehk86} and Legut \cite{l88} provided lower bounds on the maxmin value.
The optimization problem was later analysed by Dall'Aglio \cite{d02}. The general problem was reformulated as the minimization of a convex function with a particular attention to the case where the maxmin allocation is not equitable and the allocation of the cake occurs in stages to subsets of players. No detail, however, was given on how to proceed with the minimization.

In most of the fair division literature, little is assumed about the strategic behaviour of the children. Brams and Taylor \cite{bt96} discuss the issue of the manipulability of the preferences: in most cases children may benefit from declaring false preferences.
A different approach takes into account the possibility for the children to form coalitions after (Legut \cite{l90} and Legut et al. \cite{lpt94}) or before (Dall'Aglio et al.\ \cite{dbt09}) the division of the cake. In both cases coalitional games are defined and analysed. In the case of early cooperation among children, the game is based on a maxmin allocation problem among coalitions, each one having a joint utility function and a weight. The first properties of the game are studied in \cite{dbt09}. It turns out that the analysis  of the game  is made harder by the difficulties in computing the characteristic function, i.e., the value associated to each coalition. The tools we introduce, therefore, become essential in computing such values, as well as any synthetic value, such as the Shapley value, associated to the game.

The coalitional maxmin problem is indeed a generalization of the classical maxmin problem introduced by Dubins and Spanier. Therefore, we consider a common approach to set up an algorithm which, at each step, will compute an approximating allocation, together with lower and upper bounds for the maxmin value. The algorithm is based on a subgradient method proposed by Shor \cite{s85} and it yields an approximation of the optimal allocation with any fixed degree of precision.

In Section 2 we describe the maxmin fair division problem with coalitions through the strategic model of interaction
among players in \cite{dbt09} and the geometrical setting employed in \cite{{b99}}, \cite{b05} and \cite{d02}.
In Section 3 we present the upper and the lower bounds for the objective value. In Section 4 we fit the Subgradient Method to our problem and we derive a procedure where the optimal value and the optimal partition are computed up to a desired precision and we provide a numerical example where we describe two fair division games and we compute the corresponding Shapley values.  Some final considerations are given in Section 5.

\section{The model and the maxmin fair division problem with coalitions}

We represent our completely divisible good as the set $C,$ a Borel subset of $\mathbb{R}^n$,  and we denote as $\mathcal{B}(C)$ the Borel $\sigma-$algebra of subsets of $C.$
Let $N=\{1, \, \ldots, \, n\}$ be the set of players, whose
preferences on the good are $\mu_1, \, \ldots, \, \mu_n,$ where each
$\mu_i,$ $i \in N,$ is a probability measures on $(C,\mathcal{B}(C)).$
By the Radon-Nikodym theorem, if $v$ is a non-negative  finite-valued measure with respect to which each $\mu_i$ is absolutely continuous (for instance we may consider $v = \sum_{i \in N}{\mu_i}$), then, for each $A \in \mathcal{B}(C),$
\begin{equation*}
\mu_i(A) = \int_{A} f_i dv \quad \forall \; i \in N,
\end{equation*}
where $f_i$ is the Radon-Nikodym derivative of $\mu_i$
with respect to $v.$

We will consider the following assumptions:
\begin{description}
\item{a)} \emph{complete divisibility of the good (CD):} For each $i \in N$ and each $A \in \mathcal{B}(C)$
such that $\mu_i(A)> 0$, there exists a measurable set $B \subset A$ such that $\mu_i(A \cap
B)>0$ and $\mu_i(A \cap B^c)>0.$
\item{b)} \emph{mutual absolutely continuity (MAC):} If
there exists $i \in N$ and $A \in \mathcal{B}(C)$
such that $\mu_i(A) > 0,$ then for each $j \neq i$ $\mu_j(A) >0$.
\item{c)} \emph{relative disagreement (RD):}
  For each pair $i,j \in N$ and each $A \in \mathcal{B}(C)$ such that $\mu_i(A) > 0$ and $\mu_j(A) > 0$, there exists
  a measurable set $B \subset A$ such that
  $\frac{\mu_i(B)}{\mu_i(A)} \ne \frac{\mu_j(B)}{\mu_j(A)}.$
\end{description}

Throughout the rest of the work we will assume that (\emph{CD}) always holds,
while (\emph{MAC}) and (\emph{RD}) are useful, though restrictive, assumptions that we will employ only when strictly needed.

For any $h \in \mathbb{N}$, let $(A_1,\ldots,A_h)$ be an $h$-partition, i.e., a partition of the good $C$ into $h$ measurable sets. Let $\Pi_h$ be the class of all $h$-partitions.
How do players behave in the division procedure? In the simplest case,
each player competes with the others to get a part of the cake with no strategic interaction with other players. Each $(A_1,\ldots,A_n) \in \Pi_n$ determines a division of the good in which player $i$, $i \in N$ gets the share $A_i$ with value $\mu_i(A_i)$. Here, individual players seek an allocation with
values as high as possible. A fair compromise between
the conflicting interests is given by maxmin allocation
$(A^*_1,\ldots,A^*_n) \in \Pi_n$ that achieves

\begin{equation} \label{v_competitive}
v_{m} := \max_{(A_1,\ldots,A_n) \in \Pi_n} \left\{ \min_{i \in N}
\mu_i(A_i)
 \right\}.
\end{equation}
Here $v_m$ denotes the maxmin value in the classical fair division problem.
With a completely divisible good, the allocation $(A_1^*, \,
\ldots, \, A_n^*)$ is fair (or proportional), i.e.\  $\mu_i(A_i^*) \ge \frac{1}{n}$ for all $i \in N.$ Moreover, if  (\emph{MAC}) holds, it is also egalitarian, i.e.\ $\mu_i(A_i^*) = \mu_j(A_j^*) \,$ for all $i, j \in N$ (see \cite{ds61}). Therefore, under this assumption, an optimal allocation is also the bargaining solution proposed by Kalai and Smorodinsky \cite{ks75} (see also Kalai \cite{k77}).

Dall'Aglio et al.\ \cite{dbt09} proposed a strategic
model of interaction, where players, before the
division takes place, gather into mutually disjoint coalitions.
Within each coalition, players pursue an efficient allocation of
their collective share of the cake.

Let $\mathcal{G}$ be the family of
all partitions of $N$ and, for each $\Gamma \in \mathcal{G}$, let
$|\Gamma| = m$, $m \leq n,$ and let $M = \{1,\ldots,m\}$ be the coalitions indexes set.
Thus, players cluster into coalitions specified by the partition
$\Gamma=\{S_1,\ldots,S_m\}.$ For each $j \in M$ and each  coalition $S_j$, players state their joint preferences as follows
\begin{equation}\label{def_mucoal}
\mu_{S_j}(B) = \max_{\{D_i\}_{i \in S_j} \: \mathrm{partition} \: \mathrm{of} \: B} \sum_{i \in S_j} {\mu_i(D_i)} = \int_{B} f_{S_j}(x) d x
\end{equation}
with $f_{S_j}(x) = \max_{i \in S_j} f_i(x)$, $B \in \mathcal{B}(C)$ and $\{D_i\}_{i \in S_j} \in \mathcal{B}(C).$
The utility $\mu_{S_j}(B)$ of coalition $S_j$ will be divided among its members in a way that prevents any of them to break the coalition in search of a better deal.
Once the global coalition structure is known, a fair allocation of the cake among the competing coalitions is sought. In this context, assigning the same value to all coalitions could yield an unfair outcome.
Fairness here must consider the different importance that coalitions may assume and this is taken into account
by a weight function $w: 2^N \to \R_+.$

In this framework, each coalition takes the role of a single player in Equation \eqref{v_competitive}. Following  Kalai \cite{k77}, when coalitions in $\Gamma$ are formed and the weight function $w$ is considered, players should agree on a division of the cake which achieves the following value

\begin{equation}\label{v_maxmin}
v(\Gamma, w) = \max_{(B_1,\ldots,B_m) \in \Pi_m} \left\{ \min_{j \in
M } \frac{\mu_{S_j}(B_j)}{w(S_j)} \right\}.
\end{equation}

Each coalition can evaluate its performance in the division by
considering the following coalitional game
\begin{equation}
  \label{coal_game}
  \eta(S,w)=w(S) v(\Gamma_S,w)  \qquad S \subseteq N
  \end{equation}
where $\Gamma_S=\{S,\{j\}_{j \notin S}\}$. The value $\eta(S,w)$ can be interpreted as the minimal utility that coalition $S$ is going to receive in the division when the system of weight $w$ is enforced, independently of the behaviour of the other players.

A crucial question lies in the definition of the weight system. We consider two proposals:
\begin{itemize}
  \item $w_{\mathrm{card}}=|S|$, $S \subset N$. This is certainly the most intuitive setting. Although very natural, this proposal suffers from a serious drawback, since players participating in the game $\eta(\cdot,w)$ may be better off waiting to seek for cooperation well after the cake has been divided (see \cite{dbt09});
   \item $w_{\mathrm{pre}}=\mu_S\left( \cup_{i \in S} A^*_i \right)$, $S \subset N$, where $(A^*_1,A^*_2,\ldots,A^*_n) $
   is the partition maximizing \eqref{v_competitive}. By seeking early agreements among them, players will be better off than
   postponing such agreements until the cake is cut. The above mentioned problem is overcome at the cost of a less intuitive
   (and more computationally challenging) formulation (see \cite{dbt09}). It is interesting to note that to find these weights we need to solve \eqref{v_competitive}.
\end{itemize}
It is easy to verify that, for each $S \subseteq N$,

$$\eta(S, w_{\mathrm{card}}) \leq \eta(S, w_{\mathrm{pre}})$$
with equality if $S =N$ or $S= \{i\}$ where $i \in N$.

The optimization  problem \eqref{v_maxmin}
can be seen as an infinite dimensional assignment problem. In
principle we could attribute any point of the cake $C$ to any of the
participating players (provided certain measurability assumptions
are met). For very special instances this becomes a linear program:
when the preferences have piecewise constant densities, or when the
cake is made of a finite number of completely divisible and homogeneous parts.

The fully competitive value $v_m$ is a special instance of the cooperative case, since $v_m=v(\Gamma_1,w_1)$, with $\Gamma_1=\{\{1\},\ldots,\{n\}\}$ and $w_1(\{i\})=1$ for each $i \in N$. Therefore, we focus on the cooperative case alone.

\subsection{A geometrical setting}
We now describe a geometrical setting already employed in \cite{b99}, \cite{b05}, \cite{d02} and \cite{l88} to explore fair division problems. In what follows we consider the weighted preferences and
densities, $\mu_{j}^{w}$ and $f^{w}_{j}$, given respectively by
\begin{equation*}
\mu_{j}^{w}= \frac{\mu_{S_j}}{w(S_j)} \qquad f^{w}_{j}=
\frac{f_{S_j}}{w(S_j)}.
\end{equation*}
The partition range, also known as Individual Pieces Set (IPS) (see \cite{b05}) is defined as
$$\mathcal{P} := \{ (\mu_1^w(B_1), \ldots, \mu_m^w(B_m)) : (B_1, \ldots, B_m) \in \Pi_m \} \subset \mathbb{R}^m_+ .$$
Let us consider some of its features. Each point $p \in \mathcal{P}$ is the image, under $(\mu_1,\ldots,\mu_n)$, of an $m$-partition of $C$. Moreover, $\mathcal{P}$ is compact and, if (CD) holds, $\mathcal{P}$
is also convex (see \cite{l40}). Therefore, $v(\Gamma, w) = \text{max }
\{x>0 : (x, x, \ldots, x) \cap \mathcal{P} \neq \emptyset \}.$ So,
the point $v(\Gamma, w)$ is the intersection between the Pareto
frontier of $\mathcal{P}$ and the egalitarian line
\begin{equation}\label{egalitarianline}
\ell = \{x \in \mathbb{R}^m : x_1 = x_2 = \ldots = x_m \}.
\end{equation}

\section{Upper and lower bounds for the maxmin value}

We turn our attention to a simpler optimization problem
that may have an unfair solution,
but it provides easy-to-compute upper and lower bounds for the original problem.
These bounds depend on a weighted maxsum partition, which we can derive through a straightforward extension of a result by Dubins and Spanier \cite{ds61}.
Let $\Delta_{m-1}$ denote the unit $(m-1)-$simplex.

\begin{prop}(see \cite[Theorem 2]{ds61}, \cite[Proposition 4.3]{d02})
Let $\alpha \in \Delta_{m-1}$ and let
$B^{\alpha}=(B^{\alpha}_1 , \ldots, B^{\alpha}_m)$ be an $m-$partition of $C$. If 
\begin{equation}\label{ppartition}
\alpha_k f_k^w (x) \geq \alpha_h f_h^w (x) \quad \text{for all }h,k \in M \text{ and for all } x \in B^{\alpha}_k ,
\end{equation}
then
\begin{equation}
\label{prop_argmaxsum} (B_1^{\alpha},\ldots,B_m^{\alpha}) \in
\argmax_{(B_1,\ldots,B_m) \in \Pi_m} \sum_{j=1}^{m}{\alpha_j
\mu_{j}^{w}(B_j)}.
\end{equation}
\end{prop}

The value of this maxsum problem is itself an upper bound for problem \eqref{v_maxmin}.
For each choice of $\alpha \in \Delta_{m-1}$,
we have a maxsum partition $B^{\alpha}=(B^{\alpha}_1 , \ldots, B^{\alpha}_m)$ corresponding to $\alpha.$

\begin{defin}
The {\em partition value vector (PVV)} $u^{\alpha}=(u_1^{\alpha}, \ldots, u_m^{\alpha})$ is defined by

$$u_j^{\alpha} = \mu_{j}^w (B_j^{\alpha}), \quad \mbox{ for each } j = 1, \ldots, m.$$

\end{defin}

The PVV $u^{\alpha}$ is a point where the hyperplane
$\sum_{j=1}^m{\alpha_j x_j} = k$ touches the partition range
$\mathcal{P},$ so $u^{\alpha}$ lies on the Pareto border of
$\mathcal{P}.$ Moreover, for any $\alpha \in \Delta_{m-1}$ there
exists at least one PVV (see \cite{b05}). We are ready to state the
first approximation result.

\begin{prop}\label{propupper}
Let $g : \Delta_{m-1} \rightarrow \mathbb{R}^+$ be as follows:

\begin{equation*}
    g(\alpha)   :=\int_C{\max_{j \in M}{\{\alpha_j f_j^{w}(x)\}dx}}.
\end{equation*}
Then, $$v(\Gamma, w) \leq g(\alpha) \leq \max_{j \in M}{u_j^{\alpha}}.$$
\end{prop}

\begin{proof}
Following \cite[Proposition 4.3]{d02} we know that the  hyperplane that touches $\mathcal{P}$ at the point $u^{\alpha}$ is defined by the equation
$$
\sum_{i \in M} \alpha_i x_i = g(\alpha)
$$
Since $\alpha \in \Delta_{m-1}$, this hyperplane intersects the egalitarian line $\ell$ defined in \eqref{egalitarianline} at the point $(g(\alpha),\ldots,g(\alpha))$. Since the hyperplane is located above $\mathcal{P}$, this point lies above the maxmin point with coordinates $(v(\Gamma,w),\ldots,v(\Gamma,w))$. Therefore
$$
g(\alpha) \geq v(\Gamma,w)
$$

Finally, since $g(\alpha)$ is a weighted average of the values $(u_1^{\alpha}, \ldots, u_m^{\alpha}),$
it follows that $g(\alpha) \leq \max_{j \in M}{u_j^{\alpha}}.$

\qedhere
\end{proof}

The function $g$ was already considered in \cite{d02},
where it was shown that $g$ is convex, and
$v(\Gamma, w) = \min_{\alpha \in \Delta_{m-1}}{g(\alpha)}.$

We now turn our attention to a lower bound for $v(\Gamma, w).$
Although we will see later only one PVV is enough to assure such a
bound, we give a general result for the case where several PVVs have
already been computed. We derive the second approximation result
through a convex combination of these easily computable points in
$\mathcal{P},$ which lie close to $v(\Gamma, w).$ The following result generalizes Theorem 3 in \cite{l88} and Theorem 1.1 in \cite{ehk86}.

\begin{prop}\label{proplower}
Let $u = (u_1, \ldots, u_m)$ a partition value vector such that
\begin{equation} \label{max_coordinate}
u_h = \max_{j=1, \ldots, m} u_j.
\end{equation}

Then,
\begin{equation} \label{lower_bound}
v(\Gamma, w) \geq \underline{v}(u):= \frac{u_h}{1 + \sum_{j \neq h}{\frac{u_h - u_j}{\mu_{j}^{w}(C)}}} \geq \min_{j \in M}{u_j}
.
\end{equation}
\end{prop}

\begin{proof}
Let us consider the following vectors
\begin{equation}\label{e_vector}
e^q= (0, \ldots, 0, \mu_q^{w}(C), 0, \ldots, 0) \qquad q \in M \quad q \neq h,
\end{equation}
where $\mu_q^{w}(C)$ is the weighted joint utility of the whole cake by coalition $S_q.$
Now, consider the convex hull of the PVV $u$ and the $m-1$ points $e^q$, $q \neq h$,
$$V:= \{t_h u + \sum_{q \neq h}{t_q e^q}: (t_h, \ldots, t_m) \in \Delta_{m-1} \}$$
The lower bound we are looking for is the intersection point between $V$ and the egalitarian line $\ell$ from (\ref{egalitarianline}) (see Figure \ref{fig1}). Let us denote this point as $(x_w, \ldots,x_w)$.
Without loss of generality, let us suppose $h = 1.$ Then, we obtain $(x_w, \ldots,x_w)$ as follows:

\begin{displaymath}
\left\{
\begin{array}{l}
t_1 u_1 + 0 + \ldots + 0 = x_w \\
t_1 u_2 + t_2 \mu_{2}^{w}(C) + \ldots + 0 = x_w \\
\vdots \\
t_1 u_m + 0 + \ldots + t_m \mu_{m}^{w}(C) = x_w \\
t_1 + t_2 + \ldots + t_m = 1
\end{array}
\right.
\end{displaymath}

We are dealing with a linear system with $m+1$ unknown quantities, $t_1, t_2, \ldots, t_m, x_w.$

Thus, by Cramer's rule, we get $x_w$ as

\begin{align*}
x_w &= \frac{det
\begin{pmatrix}
u_1 & 0 & \ldots & 0 & 0 \\
u_2 & \mu_{2}^{w}(C) & \ldots & 0 & 0 \\
\vdots & \vdots & \vdots & \vdots & \vdots \\
u_m & 0 & \ldots & \mu_{m}^{w}(C) & 0 \\
1 & 1 & \ldots & 1 & 1
\end{pmatrix}
}{det
\begin{pmatrix}
u_1 & 0 & \ldots & 0 & -1 \\
u_2 & \mu_{2}^{w}(C) & \ldots & 0 & -1 \\
\vdots & \vdots & \vdots & \vdots & \vdots \\
u_m & 0 & \ldots & \mu_{m}^{w}(C) & -1 \\
1 & 1 & \ldots & 1 & 0
\end{pmatrix}
} 
= \frac{u_1 \prod_{q\neq 1}{\mu_{q}^{w}(C)}}{det
\begin{pmatrix}
0 & 1 & 1 & \ldots & 1 \\
-1 & u_1 & 0 & \ldots & 0 \\
- 1 & u_2 & \mu_{2}^{w}(C) & \ldots & 0 \\
\vdots & \vdots & \vdots & \vdots & \vdots \\
-1 & u_m & 0 & \ldots & \mu_{m}^{w}(C)
\end{pmatrix}
} \\
&= \frac{u_1 \prod_{q\neq 1}{\mu_{q}^{w}(C)}}{\prod_{q\neq 1}{\mu_{q}^{w}(C)} + \sum_{q \neq 1}{\prod_{i\neq q,1}{\mu_{i}^{w}(C)} (u_1 - u_q)}} \\
&= \frac{u_1 \prod_{q\neq 1}{\mu_{q}^{w}(C)}}{\prod_{q\neq 1}{\mu_{q}^{w}(C)} \left[ 1 + \sum_{q \neq 1}{\frac{(u_1 - u_{q})}{\mu_{q}^{w}(C)}} \right]} 
= \frac{u_1}{1 + \sum_{q \neq 1}{\frac{(u_1 -u_{q})}{\mu_{q}^{w}(C)}}},
\end{align*}
where the second equality derives by suitable exchanges of rows and columns in the denominator matrix. 
In fact, we get the second one after an even number of exchanges on the first: 
$m$ successive exchanges of the last row until it reach the first position, 
and $m$ successive exchanges of the last column until it reach the first position. 
So the two matrices in the denominator have the same determinant.
It is easy to verify that $t_i > 0$, for every $i \in N$.

Finally, since the lower bound belongs to the convex hull of the PVV $u$ and the $m-1$ vectors $e^q,$ $q \neq h$,
it is not less than the minimum component of each vector,
in particular $\underline{v}(u) \geq \min_{j \in M} {u_j}.$

\qedhere
\end{proof}

An illustration of the position of the bounds with respect to the partition range in the case of two coalitions is shown in Figure \ref{fig1}.
\begin{figure}[htp]
\centering
\includegraphics[scale=0.20]{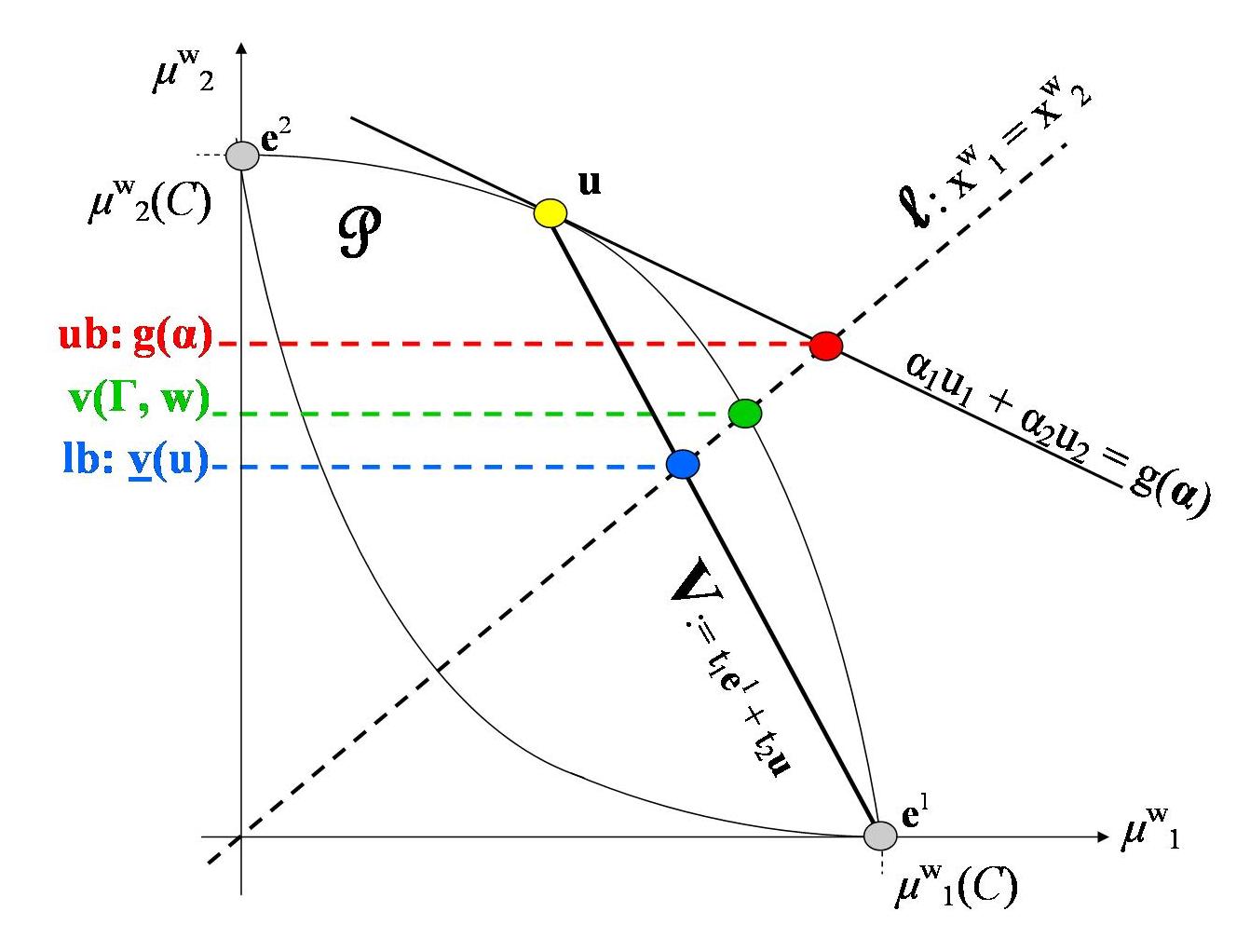}
\caption{Upper and lower bounds for the two-coalition
case.}\label{fig1}
\end{figure}

\section{The subgradient method}

In the previous section we have seen that for each choice of the
coefficients vector $\alpha$ we can derive upper and lower bounds for
$v(\Gamma, w).$ We describe a way of improving the coefficients
$\alpha$ so that eventually the bounds will shrink to the desired value.

Since in general $g$ is a non-differentiable convex function, we can rely on a simple minimizing algorithm developed by Shor \cite{s85}, the subgradient method.
In particular, since the domain of $g$ is constrained, we must consider an extension, the projected subgradient method, which solves constrained convex optimization problems.
Let us start by describing the method through some basic definitions and the essential convergence result.

\begin{defin}
Let $D$ be a closed convex set and let $|| \cdot ||$ be the Euclidean norm. The projection of $z \in \mathbb{R}^n$ on $D$ is denoted by $p(z)$ and it is defined as
\begin{equation} \label{def_projection}
p(z) = \argmin_{x \in D} || x - z ||.
\end{equation}
\end{defin}

\begin{defin} Let $f$ be a convex function with domain $D$ and let $x_0$ an interior point of $D$.
A vector $\gamma(x_0)$ is called a subgradient or a generalized gradient of $f$ at $x_0$ if it satisfies

\begin{equation} \label{subgradient}
f(x) - f(x_0) \geq \langle \gamma(x_0), x - x_0 \rangle \quad \text{for all } x \in D.
\end{equation}

Moreover, $\gamma$ is a bounded subgradient of $f$
if there exists $G \in \mathbb{R}^+$ such that 
$||\gamma(x)|| \leq G$ for all $x \in D.$
\end{defin}

We denote as $\partial_x f(x)$ the set of subgradients of a convex function $f$ at any interior point $x$ of the $f$ domain.

\begin{defin}
A sequence $\{s_t\}_{t=0}^{+ \infty}$ of positive numbers  is called {\em diminishing step size rule} if it satisfies the conditions:
\begin{gather}
\label{dss1}
    \lim_{t \to + \infty} s_t = 0,
\\
\label{dss2}
    \sum_{t=0}^{+ \infty}{s_t} = + \infty.
\end{gather}

\end{defin}

The subgradient method minimizes a non-differentiable convex function which has a bounded set of minimum points and at least one bounded subgradient. 
This procedure returns a minimum value for the function 
moving a point in the domain in the opposite direction of 
a bounded subgradient by a step belonging to a diminishing step size rule (see \cite{s85}). 
If the domain of the function is constrained, then the point is projected in the domain (see \cite{bxm03}).
We recall the general result
\begin{prop}\label{propsgconv} (see  \cite{bxm03}, \cite{s85})
Let $f$ be a convex function defined on $D \subseteq \mathbb{R}^m,$
which has a bounded set of minimum points $D^*$ and let $\gamma(x) \in
\partial_x f(x)$ be a bounded subgradient. 
Moreover, let
$\{s_t\}_{t=0}^{+ \infty}$ be a diminishing step size rule.
Then for any $x^0 \in D$ the sequence $\{x^t\}_{t=0}^{+ \infty}$ generated according to the formula

\begin{equation} \label{projection}
x^{t+1} = p[x^t - s_t \gamma(x^t)]
\end{equation}
has the following property: either an index $t^*$ exists such that $x^{t^*} \in D^*,$
or $\lim_{t \to + \infty} f_{best}^t - f^* = 0,$ where
$$
    f_{best}^t = \min_{i=1, \ldots, t}f(x^t) \qquad \mbox{and} \qquad
    f^* = \min_{x \in D} f(x).
$$
\end{prop}

Let us check that $g$ can be minimized through the projected
subgradient method. First of all, $g$ is convex with
$\min_{\alpha \in \Delta_{m-1}} g(\alpha)= v(\Gamma, w)$ (see \cite{d02}) and we can easily
show that $u^{\alpha} \in \partial_{\alpha} g(\alpha),$ with $u^{\alpha}$ bounded. 
In fact, for each point $\tilde{\alpha} \in \Delta_{m-1}$
the vector $u^{\tilde{\alpha}}$ satisfies \eqref{subgradient}:

\begin{align*}
g(\alpha) - g(\tilde{\alpha}) &= \max_{B \in \Pi_m} \sum_{j \in M}{\alpha_j \mu_j^w(B_j)} - \langle u^{\tilde{\alpha}}, \tilde{\alpha} \rangle \\
&\geq \sum_{j \in M}{\alpha_j \mu_j^w(B_j^{\tilde{\alpha}})} - \langle u^{\tilde{\alpha}}, \tilde{\alpha} \rangle \\
&= \langle u^{\tilde{\alpha}}, \alpha \rangle - \langle u^{\tilde{\alpha}}, \tilde{\alpha} \rangle \\
&= \langle u^{\tilde{\alpha}}, \alpha - \tilde{\alpha} \rangle .
\end{align*}

We now adapt the general updating rule \eqref{projection} to our situation. For any diminishing step size rule $\{s_t\}_{t=0}^{+\infty}$ and any vector $\alpha^{t} \in \Delta_{m-1}$ of coefficients, the update rule becomes
\begin{equation}
\label{updaterule}
\alpha^{t+1}= p[\alpha^{t} - s_t u^{t}] = (\alpha^{t} - s_t u^{t} + \lambda \boldsymbol{1})_+
\end{equation}
where $\lambda \in \mathbb{R}$ is the normalizing constant such that
$$\sum_{i=1}^m{(\alpha_i^t - s_t u_i^t + \lambda)_+} =1.$$
Suppose now that $\alpha^t \in \stackrel{\circ}{\Delta}_{m-1}$ and that the step size $s_t$ is sufficiently small to guarantee
\begin{equation}
\label{sssmall} \alpha_i^t - s_t u_i^t + \lambda > 0 \quad \mbox{ for each } i \in M.
\end{equation} 
Here $\lambda$ has to be chosen so that
$$\sum_{i=1}^m{(\alpha_i^t - s_t u_i^t + \lambda)} =1, $$ i.e.,
$$\sum_{i=1}^m{\alpha_i^t} - s_t \sum_{i=1}^m{u_i^t} + m\lambda = 1,$$ hence
$$ \lambda = s_t \bar{u}^t, $$
where $\bar{u}^t= \frac{\sum_{i=1}^m{u_i^t}}{m}$ is the average of
the subgradient vector components.
In what follows we will make sure to choose a diminshing step size rule small enough so that \eqref{sssmall} is verified, or, equivalently,
\begin{equation}
\label{newdssr}
\alpha_i^t - s_t(u^t_t - \bar{u}^t) >0 \mbox{ for all } t \in \mathbb{N} \quad \mbox{ and for all }  i = 1, \ldots, m.
\end{equation}

We are now able to state the first convergence result.
\begin{prop}\label{propsprime}
Suppose (CD) and (MAC) holds. Let $\Delta^*_{m-1}$ be the bounded set of minimum points for $g$ and let $\{s_t\}_{t=0}^{+ \infty}$ be a diminishing step size rule.
Then, there exists another diminishing step size rule $s^{'}_t \leq s_t$ which satisfies \eqref{newdssr}.
Consequently, given $\alpha^0 \in \Delta_{m-1}$ and the recursive sequence
\begin{align}
u^t &= PVV(\alpha^t) \notag \\
\alpha^{t+1} &= \alpha^t - s'_t(u^t - \bar{u}^t) \label{recalpha}
\end{align}
either $\alpha^{t^*} \in \Delta^*_{m-1}$ for some $t^* \in \mathbb{N}$, or
$$\lim_{t \to + \infty} \alpha^t = \alpha^* \in \Delta^*_{m-1} \qquad \mbox{and} \qquad
\lim_{t \to + \infty} g(\alpha^t) = g(\alpha^*) = v(\Gamma,w).$$
\end{prop}

\begin{proof}

First of all, notice that constraint \eqref{newdssr} involves only those indexes $i \in \{1, \ldots, m\}$ for which $u_i^t > \bar{u}^t.$ Let be $\mathcal{I}_t$ the set of those indexes in the step $t.$ For each of them we would get
$$s_t < \frac{\alpha_i^t}{u_i^t - \bar{u}^t} = \frac{m \alpha_i^t}{(m-1)u_i^t - \sum_{j \neq i}{u_j^t}}.$$
Now, fix an arbitrary integer $K \in \mathbb{N}$ and define
$$\tau_t = \frac{K-1}{K} \min_{i \in \mathcal{I}_t}{\left\{\frac{m \alpha_i^t}{(m-1)u_i^t - \sum_{j \neq i}{u_j^t}}\right\}}.$$
Hence, let us define
$$s^{'}_t= \min\{s_t, \tau_t \}.$$

Thus, $s^{'}_t$ satisfies \eqref{newdssr} and \eqref{dss1}, since $\lim_{t \to + \infty} s^{'}_t = \lim_{t \to + \infty} s_t=0.$ To show \eqref{dss2}, let us suppose $\sum_{t=0}^{+ \infty}{s^{'}_t} < + \infty.$
This implies for some $i^* \in \{1, \ldots, m\}$ and some sequence $\{t_r\} \subset \mathbb{N}$

$$ \lim_{t_r \to + \infty} \frac{m \alpha_{i^*}^{t_r}}{(m-1)u_{i^*}^{t_r} - \sum_{j \neq i^*}{u_j^{t_r}}} = 0.$$
Since $(m-1)u_{i^*}^{t_r} - \sum_{j \neq i^*}{u_j}^{t_r} > 0,$ taking a further subsequence $\{t_p\} \subset \{t_r\}$ we have
\begin{align}
&\alpha_{i^*}^{t_p} \rightarrow \tilde{\alpha}_{i^*} = 0, \text{ so } \sum_{j \neq i^*}{\alpha_j^{t_p}} \rightarrow \sum_{j \neq i^*}{\tilde{\alpha}_j} = 1, \text{ and} \notag \\
&u_j^{t_p} \rightarrow \tilde{u}_j \text{ for all } j \in \{1, \ldots, m\}, \notag \\
&\text{with } (m-1)\tilde{u}_{i^*} > \sum_{j \neq i^*} {\tilde{u}_j} \label{sum}.
\end{align}
By continuity, $\tilde{u}= (\tilde{u}_1, \ldots, \tilde{u}_m)$
lies on the upper surface of $\mathcal{P},$ so
\begin{equation}\label{uoptimal}
\sum_{j \in M}{\tilde{u}_j} \geq 1.
\end{equation}
Moreover, $\tilde{u}$ is supported by the hyperplane $\sum_{j=1}^m{\tilde{\alpha}_j x_j} = k.$

By (\ref{sum}) and (\ref{uoptimal}) we have that $\tilde{u}_{i^*} \geq \frac{1}{m} >0,$ since
\begin{align*}
(m-1)\tilde{u}_{i^*} &> \sum_{j \neq i^*} {\tilde{u}_j} = \sum_{j \in M}{\tilde{u}_j} - \tilde{u}_{i^*} \geq 1 - \tilde{u}_{i^*}, \text{ so} \\
\tilde{u}_{i^*} &\geq \frac{1}{m} > 0.
\end{align*}

Now, the coexistence of $\tilde{u}_{i^*} > 0$ and
$\tilde{\alpha}_{i^*} = 0$ clashes with the hypothesis (MAC). In
fact, $(\tilde{u}_1, \ldots, \tilde{u}_m) \in \argmax_{x \in
\mathcal{P}} \sum_{j \in M}{\tilde{\alpha}_j x_j},$ $\sum_{j \in
M}{\tilde{\alpha}_j \tilde{u}_j} = \sum_{j \neq
i^*}{\tilde{\alpha}_j \tilde{u}_j}= k$ and there is no $(x_1,
\ldots, x_m) \in \mathcal{P}$ for which $\sum_{j \in
M}{\tilde{\alpha}x_j}>k.$ Since $\tilde{u}_{i^*} \geq \frac{1}{m}
>0,$ there exists $\tilde{A}_{i^*}$ such that $\tilde{u}_{i^*} =
\mu_{i^*}(\tilde{A}_{i^*}) \geq \frac{1}{m}.$ By (MAC) we can derive
a partition from $\tilde{A}_{i^*}$ of $(m-1)$ subsets $\{B_j\}_{j
\neq {i^*}},$ with $\cup_{j \neq i^*} B_j = \tilde{A}_{i^*}$ and
$B_j \cap B_l = \emptyset$ if $j \neq l,$ such that $\mu_j (B_j)
\geq \varepsilon >0$ for all $j \neq i^*.$

If we consider the partition $\tilde{\tilde{A}}$ defined as
$\tilde{\tilde{A}}_{i^*} = \emptyset,$ $\tilde{\tilde{A}}_j =
\tilde{A}_j \cup B_j,$ we get
$$\sum_{j \in M}{\tilde{\alpha}_j} \mu_j(\tilde{\tilde{A}}_j) = \sum_{j \neq i^*}{\tilde{\alpha}_j(\mu_j(\tilde{A}_{i^*}) + \mu_j(B_j))} = k + (m-1)\varepsilon > k,$$
which is a contradiction.

The final statement is a direct consequence of Propositions \ref{propupper}, \ref{propsgconv} and of the fact that if $s'_t \to 0$ then the sequence $\{\alpha^t\}$ must converge to some $\alpha^* \in \Delta^*_{m-1}$. Moreover, by the dominated convergence theorem,  $g(\alpha^t) \to g(\alpha^*) = v(\Gamma,w)$, last equality being again a consequence of Proposition \ref{propupper}.
\qedhere
\end{proof}

To prove the convergence of the PVVs and of the lower bound, we assume \emph{relative disagreement} (RD).
\begin{prop} \label{convlb}
  Suppose (CD), (MAC) and (RD) hold. Then, for any $\alpha^0 \in \Delta_{m-1}$ and the recursive sequence \eqref{recalpha}, one of the following two conditions hold:
  \begin{itemize}
    \item either $u^{t^*}=(v(\Gamma,w),\ldots,v(\Gamma,w))$ for some $t^* \in \mathbb{N}$ and $\underline{v}(u^{t^*}) = v(\Gamma,w)$,
        \item or $\lim_{t \to +\infty} u^t =  (v(\Gamma,w),\ldots,v(\Gamma,w))$ and $\underline{v}(u^t) \to v(\Gamma,w)$.
  \end{itemize}
\end{prop}

\begin{proof}
By (RD) we have that for any point on Pareto border of $\mathcal{P}$ there  exists one and only one hyperplane touching $\mathcal{P}$ (see \cite{b05}).
By the conclusions of Proposition \ref{propsprime} either $\alpha^{t^*} \in \Delta^*_{m-1}$ or $\lim_{t \to + \infty} \alpha^t = \alpha^* \in \Delta^*_{m-1}.$

In the first case, there exists only one PVV $u^{t^*}$ corresponding to $\alpha^{t^*}.$ Since the hyperplane with coefficients vector $\alpha^{t^*}$ touches the partition range $\mathcal{P}$ in the point corresponding to the maxmin allocation, then $u^{t^*}$ must coincide with
\begin{equation}\label{PVVegal}
u^*=(v(\Gamma,w),\ldots,v(\Gamma,w)).
\end{equation}
Also $\underline{v}(u^*)=v(\Gamma,w).$ In fact, all the coordinates of $u^*$ are equal and, therefore, maximal. Without loss of generality we choose $u_1$ as  maximal, and
$$
\underline{v}(u^*) = \frac{u_1}{1 + \sum_{j \neq 1} \frac{u_1 - u_j}{\mu^w_j(C)}} = u_1 = v(\Gamma,w)
$$
the last equality holding by \eqref{PVVegal}.

In the second occurrence, suppose on the contrary that $\alpha^{t^*} \to \alpha^*$ while $u^t \nrightarrow u^*$. Since the sequence $\{u^t\}$ is in a compact set, there must be a converging subsequence $u^{t'} \to \tilde{u} \neq u^*$. The vector $\tilde{u}$ is a second PVV associated to $\alpha^*$, but this is ruled out by (RD). Thus, $\lim_{t \to +\infty} u^t = u^*,$ and, by continuity, $\lim_{t \to + \infty} \underbar{v}(u^t) = \underbar{v}(u^*) = u^*$.
\qedhere
\end{proof}

\subsection{The algorithm}
We now present two versions of an algorithm for the maxmin division problem. The common initializing elements for both versions are listed in Table \ref{algelement}. The first version computes upper and lower bounds for $v(\Gamma,w)$ and updates the coefficient vector $\alpha$ through the subgradient rule \eqref{updaterule}. Both bounds are updated by means of a simple comparison with the old ones. The generic step is described in Table \ref{algstep}. A simpler but slower version, described in Table \ref{alg2step}, computes the approximating optimal partition as well as the value. The finiteness of both algorithms is guaranteed by Propositions 
\ref{propsgconv}, \ref{propsprime} and \ref{convlb}.

Particular care is needed in choosing the step sequence $\{s_t\}_{t=0}^{+ \infty}$. A sequence converging too fast to 0 may lead to an increase in the  number of steps needed, since the step may soon become too small to reach the optimum. Similarly a sequence converging too slowly may result in values of $\alpha$ (and of the corresponding PVV's) jumping from one extreme to the other of the unit simplex (and of the partition range for the PVV's). This, again, will slow the convergence process.

\begin{table}[h]
\caption{Description and initialization algorithms elements.}
\label{algelement}
\centering
	\begin{tabular}{cll}
		\toprule
		Elements & Description & Initialization\\
		\midrule
		$\alpha$ & supporting hyperplane coefficients vector & $\alpha_j^0={1}/{m}$, $j \in M$ \\
		$u$ & PVV vector associated to $\alpha$ & $u^0 = u^{\alpha^0}$ \\
		$ub$ & upper bound & $g(\alpha^0)$\\
		$lb$ & lower bound & $\underline{v}(u^0)$\\
		$s_t^{'}$ & diminishing step size rule & $s_0^{'} = \min\{s_0, \tau_0\}$ \\
		\bottomrule	
	\end{tabular}
\end{table}

\begin{table}[h]
\caption{Generic step $t$ for algorithm returning $v(\Gamma, w)$.}
\label{algstep}
\centering
	\begin{tabular}{lp{0.75\columnwidth}}
		\toprule
		Step $t$ & Computation \\
		\midrule
		update $\alpha$ & $\alpha^t = \alpha^{t-1} - s^{'}_t(u^{t-1} - \bar{u}^{t-1})$ \\
		update $u$ & $u^{t}=PVV(\alpha^t)$ \\
		update $ub$ & if $g(\alpha^t) < ub,$ then $ub= g(\alpha^t)$ \\
		update $lb$ & if $\underline{v}(u^t) > lb,$ then $lb = \underline{v}(u^t)$\\
		stop condition & if $ub - lb < \varepsilon,$ then STOP. Else, repeat the step. \\
	\bottomrule	
	\end{tabular}
\end{table}

\begin{table}[h]
\caption{Generic step $t$ for algorithm returning $u^*$.}
\label{alg2step}
\centering
	\begin{tabular}{lp{0.75\columnwidth}}
		\toprule
		Step $t$ & Computation \\
		\midrule
		update $\alpha$ & $\alpha^t = \alpha^{t-1} - s^{'}_t(u^{t-1} - \bar{u}^{t-1})$ \\
		update $u$ & $u^{t}=PVV(\alpha^t)$ \\
		stop condition & if $\max_{j \in M}{u_j^{t}} - \min_{j \in M}{u_j^{t}} < \varepsilon,$ then STOP. Else, repeat the step. \\
	\bottomrule	
	\end{tabular}
\end{table}

\subsection{A five players example}
Let us consider the coalitional game defined in \eqref{coal_game},
with five players and players' preferences listed as probability distributions on $C=[0,1]$ in Table \ref{distributions}.

\begin{table}[ht]
\caption{Players preferences.}
\label{distributions}
\centering
	\begin{tabular}{cl}
		\toprule
	Player $\{i\}$ & $\mu_i$ \\
		\midrule
  1 & $\mbox{Beta} (2,5)$ \\
  2 & $\mbox{Beta} (3,8)$ \\
  3 & $\mbox{Beta} (7,2)$ \\
  4 & $\mbox{Beta} (10,10)$ \\
  5 & $\mbox{Uniform} [0,1]$ \\
	\bottomrule	
	\end{tabular}
\end{table}

In Figure \ref{densitiespartition}, we represent the initial densities (a) and then the maxmin partition for the fully competitive context (b), where $\Gamma = \{\{i\}_{i \in N}\}$ and $w(\{i\})=1,$ for all $i \in N.$

\begin{figure}[h]
\centering
\subfloat[][\emph{Densities of preferences}.]
{\includegraphics[width=0.45\columnwidth]{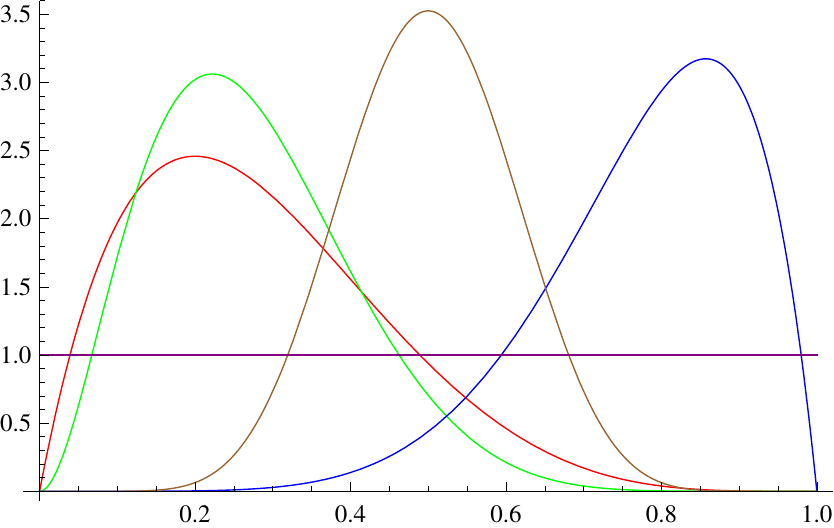}} \quad
\subfloat[][\emph{Maxmin partition}.]
{\includegraphics[width=0.45\columnwidth]{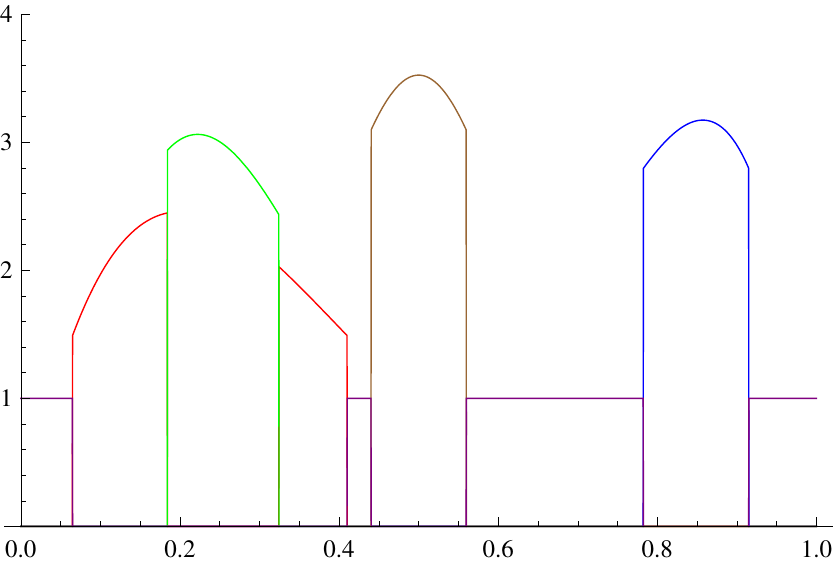}} \\
\caption{Densities of players preferences and maxmin partition in competitive context (red for player 1, green for player 2, blue for player 3, brown for player 4, purple for player 5).}\label{densitiespartition}
\end{figure}

For any $S \subseteq N$, we run our algorithm enforcing the two
weight systems $w_{card}$ and $w_{pre}$,
with a tolerance of $10^{-3}$ and we compute the corresponding game values (Table \ref{game_value}).
Consequently, in Table \ref{shapley_value}, we compute the Shapley value for each game.

\begin{table}[!ht]\footnotesize
\caption{Comparison between the coalitional game values.} 
\label{game_value}
\centering
\[	
	\begin{array}{lcc}
		\toprule
		\text{Coalition } S & \eta(S,w_{card}) & \eta(S,w_{pre}) \\
		\midrule
		\{i\}, i=1, 2, 3, 4, 5 & 0,404 & 0,404 \\
		\{1,2\} & 0,822 & 0,842 \\
    	\{1,3\} & 0,835 & 0,836 \\
		\{1,4\} & 0,844 & 0,861 \\
  		\{1,5\} & 0,819 & 0,827 \\
  		\{2,3\} & 0,820 & 0,820 \\
  		\{2,4\} & 0,826 & 0,826 \\
  		\{2,5\} & 0,828 & 0,833 \\
  		\{3,4\} & 0,808 & 0,808 \\
  		\{3,5\} & 0,926 & 1,040 \\
  		\{4,5\} & 0,886 & 1,004 \\
  		\{1,2,3\} & 1,262 & 1,280 \\
  		\{1,2,4\} & 1,273 & 1,302 \\
  		\{1,2,5\} & 1,256 & 1,265 \\
  		\{1,3,4\} & 1,275 & 1,289 \\
  		\{1,3,5\} & 1,392 & 1,465 \\
  		\{1,4,5\} & 1,366 & 1,427 \\
  		\{2,3,4\} & 1,242 & 1,241 \\
  		\{2,3,5\} & 1,389 & 1,474 \\
  		\{2,4,5\} & 1,349 & 1,414 \\
  		\{3,4,5\} & 1,403 & 1,625 \\
  		\{1,2,3,4\} & 1,706 & 1,727 \\
  		\{1,2,3,5\} & 1,877 & 1,903 \\
  		\{1,2,4,5\} & 1,841 & 1,862 \\
  		\{1,3,4,5\} & 1,968 & 2,044 \\
  		\{2,3,4,5\} & 1,940 & 2,032 \\
    	\{1,2,3,4,5\} & 2,477 & 2,477 \\
	\bottomrule	
	\end{array}
\]
\end{table}

\begin{table}[!h]\small
\caption{Comparison between the Shapley values.} 
\label{shapley_value}
\centering
\[
	\begin{array}{ccc}
		\toprule
	\text{Player} \{i\} & \text{S.V. in } \eta(S,w_{card}) &  \text{S.V. in } \eta(S,w_{pre}) \\
		\midrule
  1 & 0,465 & 0,436 \\
  2 & 0,451 & 0,425 \\
  3 & 0,507 & 0,519 \\
  4 & 0,491 & 0,502 \\
  5 & 0,563 & 0,594 \\
	\bottomrule	
	\end{array}
\]
\end{table}

The two games share the same ranking for the Shapley values
$$
\mbox{Pl.5} \succ \mbox{Pl.3} \succ \mbox{Pl.4} \succ \mbox{Pl.1} \succ \mbox{Pl.2}
$$
which therefore seems to be robust enough to the choice of system weights.

Also, the weight system $w_{pre}$ amplifies the difference in the Shapley values obtained with $w_{card}$, yielding a higher variance for the values' distributions.

\section{Concluding remarks}

In the previous section we described a couple of algorithms that return maxmin values and partitions in both competitive and cooperative settings. It is important to note that we could think of the same procedures as interactively implemented between (coalitions of) players and an impartial referee. At first the referee proposes a division of the cake based on the maxsum division of the cake with equal weights for all players. The players now report their utilities and the referee corrects the inequalities in the division by proposing a new maxsum division with modified weights: Players who were better off will be given a smaller weight and those who were worst off will see their weight increase. Of course, one cannot hope to achieve the same degree of precision, since the algorithm performs that step dozens of times, but the bounds described in Section 3 give  a precise idea on how far the proposed division is from the desired one.

Many issues remain open. We hint at two of them.
\begin{itemize}
\item In the numerical example it would be interesting to link the Shapley value rankings to the original system of preferences. What makes Players 5 and 3 the most powerful players in the cooperative division process? Apparently the two utility functions have different features: Player 5's distribution is uniform over the unit interval and his density is maximal only at the very ends of the interval. On the other hand, Player 3's preferences are concentrated at the second half of the interval -- where he has no competitors, except player 5 (who, however, has a smaller density). No simple explanation could be provided so far.
\item Beyond the convergence of the algorithms, which end in a finite number of steps, returning the approximate solution up to a specified degree of precision, it would be interesting to investigate about the computational efficiency of the same algorithms

\end{itemize}

\end{document}